\documentclass[12pt]{amsart}
\usepackage{amscd,amssymb,verbatim}
\usepackage{amssymb,amsmath,amsthm,epsfig, mathrsfs}
\usepackage{amsfonts,amsbsy,upref}
\usepackage[mathscr]{eucal}

\def\cnum#1{\bigcirc\kern -8pt#1}

\theoremstyle{plain}
\newtheorem{thm}{Theorem}[section]
\newtheorem{pro}[thm]{Proposition}
\newtheorem{lem}[thm]{Lemma}
\newtheorem{cor}[thm]{Corollary}

\newtheorem{definition}{Definition}[section]

\newtheorem{example}[thm]{Example}
\newtheorem*{problems}{Problems}

\theoremstyle{definition}

\newtheorem{rem}{Remark}

\newcommand{\co}{\rm{co}}

\begin{document}

\title{On the approximate fixed point property in abstract spaces}

\author{C. S. Barroso}
\address{Departamento de Matem\'atica,
Universidade Federal do Cear\'a, Bl 914, 60455-760, Campus do Pici, Fortaleza-CE, Brazil}
\email{cleonbar@mat.ufc.br}

\thanks{C. S. Barroso's research has been partially supported by the Brazilian-CNPq Grant.}

\author{O. F. K. Kalenda}
\address{Department of Mathematical Analysis, Faculty of Mathematics and Physics, Charles University, Sokolovsk\'a 83, 186 75, Praha 8, Czech Republic}
\email{kalenda@karlin.mff.cuni.cz}
\thanks{O. F. K. Kalenda supported in part by the grant GAAV IAA 100190901  and in part by the Research Project MSM~0021620839 from the Czech Ministry of Education.}

\author{P.-K. Lin}
\address{Department of Mathematics, University of Memphis, Memphis, TN 38152, USA}
\email{pklin@memphis.edu}

\thanks{*A very preliminary version of this work was presented at the 4th ENAMA (National Meeting on Mathematical Analysis and Applications) in Bel\'em-PA, Brazil.}

 \markboth{C. S. Barroso, O. F. K. Kalenda \and P.-K. Lin}{Approximate fixed point property}

\date{\today}


\begin{abstract} Let $X$ be a Hausdorff topological vector space, $X^*$ its topological dual and $Z$ a subset of $X^*$. In this paper, we establish some results concerning the $\sigma(X,Z)$-approximate fixed point property for bounded, closed convex subsets $C$ of $X$. Three major situations are studied. First when $Z$ is separable in the strong topology. Second when $X$ is a metrizable locally convex space and $Z=X^*$, and third when $X$ is not necessarily metrizable but admits a metrizable locally convex topology compatible with the duality. Our approach focuses on establishing the Fr\'echet-Urysohn property for certain sets with regarding the $\sigma(X,Z)$-topology. The support tools include the Brouwer's fixed point theorem and an analogous  version of the classical Rosenthal's $\ell_1$-theorem for $\ell_1$-sequences in metrizable case. The results are novel and generalize previous work obtained by the authors in Banach spaces.


\end{abstract}

\maketitle


\section{Introduction}
\label{sec:Int}
Let $X$ be a Hausdorff topological vector space (abbreviated HTVS) and $C$ a bounded, closed convex subset of $X$. In this paper we address the problem whether every continuous mapping $f\colon C\to C$ has in some sense an approximate fixed point sequence, that is, a sequence $(x_n)_n$ such that $(x_n - f(x_n))_n$ converges to zero as $n$ goes to infinity. This problem has been investigated by several authors, and very often there have been a number of works using results along this line to address various problems arising in several branches of mathematics. We would like to address the reader to \cite{Branzei1, Branzei} and references therein where metric approximate fixed point results are used for solving important problems in game theory, including approximate Nash equilibrium in games. The study of this problem in the general framework of HTVS is closely related to Almost-Fixed Point Theory which was apparently started by Walt \cite{Walt}, van de Vel \cite{Vel}, Hazewinkel and van de Vel \cite{Hazewinkel-Vel}, and Idzik \cite{Idzik1,Idzik2}. Nevertheless, our greatest motivation comes from some studies done in the setting of normed and Banach spaces, which is the subject matter of the present paper and brings out the necessity of considering weaker topologies ensuring the sequential approximation of fixed points where no stronger convergence can be expected. For example, when $X$ is a normed space with norm $\|\cdot\|$, Lin-Sternfeld \cite{Lin-Sternfeld} proved that if $C$ is not totally bounded then there exists a Lipschitz mapping $f$ from $C$ into itself such that $\inf_C\|x-f(x)\|>0$. In particular, if $X$ is Banach and $C$ is noncompact then such a mapping can always be constructed. It is natural then to look for weak-approximating fixed point sequences instead strong one. When $X$ is a Banach space, a nonexistence result was reported by Dom\'\i nguez Benavides, Jap\'on Pineda and Prus in \cite{BPP}. They proved that every closed convex  subset of a Banach space which is not weakly compact contains a closed convex subset $K$ which fails certain kind of approximation of fixed points for continuous affine self-mappings (more precisely, there is a continuous affine map $T:K\to K$ such that $\inf\{\liminf \|y-T^n x\|:x,y\in K\}>0$). In \cite[Lemma 1]{Moloney-Weng} Moloney and Weng proved that if $X$ is a Hilbert space and $C$ is a closed ball, then every demicontinuous mapping $f\colon C\to C$ admits a weak approximated fixed point sequence, that is, a sequence $\{x\sb n\}\subset C$ such that $(x\sb n - f(x\sb n))\sb n$ converges weakly to zero. This result was used in the proof of a fixed point theorem for demicontinuous pseudocontractions self-mapping closed, bounded, convex subsets of Hilbert spaces. Motivated by these results, Barroso \cite{Barroso} started the study of the weak-approximate fixed point property (weak-AFPP, in short) in Banach and abstract spaces. Let $X^*$ denote the topological dual of $X$.

\begin{definition} Given a subspace $Z$ of $X^*$, we say that $C$ has the $\sigma(X,Z)$-approximate fixed point property (AFPP, in short) if for every continuous mapping $f$ from $C$ into itself, there exists a sequence $(x\sb n)_n$ in $C$ so that $(x^*(x\sb n - f(x\sb n)))_n$ converges to zero for all $x^*\in Z$. Similarly, we say that $X$ has the $\sigma(X,Z)$-AFPP if every bounded, closed convex subset of $X$ has the property.
\end{definition}

\begin{rem}
When $Z=X^*$ we simply write weak-AFPP instead writing $\sigma(X,X^*)$-AFPP. In a similar way, we can also define the $\sigma(X^*,Z)$-AFPP for some subset $Z$ of $X$.
\end{rem}

The main result in \cite[Theorem 3.1]{Barroso} concerning this topic in Banach spaces can be stated as: Every weakly compact convex subset of a Banach space has the weak-AFPP for norm-continuous maps. We should mention at this point that the weak-AFPP features a close relation with some geometric aspects of Banach spaces. For instance, by \cite{Lin-Sternfeld} we can conclude that $X$ will not have the weak-AFPP if it contains an isomorphic copy of $\ell\sb 1$. In fact, by Rosenthal's $\ell\sb 1$-theorem every Schur space fails to have this property. In contrast, subsequent to \cite{Barroso}, it was proved by Barroso and Lin \cite[Theorem 2.2]{Barroso-Lin} that every Asplund space has the weak-AFPP. Even though the class of Asplund spaces encompass a huge variety of spaces including reflexive spaces \cite{Troyanski}, spaces with separable dual \cite{Asplund}, $C(K)$-spaces with $K$ scattered and many other (see e.g. Namioka-Phelps \cite{Namioka-Phelps} for more details on this subject), in \cite{Barroso-Lin} it was posed the question of whether a Banach space not containing $\ell\sb 1$ isomorphically should have the weak-AFPP. Very recently, Kalenda \cite{Kalenda} solved this question using a powerful theorem of Bourgain, Fremlin and Talagrand (\cite[Theorem 3F]{BFT}), so that a complete characterization of weak-AFPP is obtained: a Banach space has the weak-AFPP if and only if it contains no isomorphic copy of $\ell\sb 1$. This provides a quite interesting result that contrasts with many previous studies concerning the fixed point property (FPP) where, at least for nonexpansive maps, no characterization of FPP or weak-FPP seems to be known (cf. \cite{Benavides} and references therein).

The aim of this paper is to continue the study begun in \cite{Barroso, Barroso-Lin} which naturally leads to a number of related questions which are interesting on their own. We hope that the present study will shed some light on the issue of whether one can get a characterization of weak-AFPP for HTVS like that for Banach spaces. To put things into some perspective, we should mention that all our results here are either new, more general or sharper than similar ones. Next we give an outline of the paper: It is divided into three sections. In Section 2, we initially set up the basic framework four our results throughout the paper. Next we establish our first main result (Theorem \ref{prop:1sec2}). It generalizes the main result of \cite{Barroso-Lin} for HTVS with separable strong dual, sets not necessarily closed and not necessarily continuous maps. As a byproduct, we get an easy proof of a fixed point result due to Ky Fan (see Corollary \ref{cor:1sec2}). In Proposition \ref{prop:sepdual}, we indicate some circumstances where Theorem \ref{prop:1sec2} can be applied. Moreover, two illustrative examples concerning the $\sigma(X^*,X)$-AFPP are included. In Section 3, we study the weak-AFPP for the case when $X$ is a metrizable locally convex space (LCS, in short). The idea is to adapt in an accurate way the ingredients used in \cite{Kalenda} to this context. Firstly, we introduce the notion of $\ell_1$-sequences in topological vector spaces
 which reduces to that of isomorphic copies of $\ell_1$ in case of Banach spaces (see Definition \ref{def:1sec3}). After, we shall use a characterization of $\ell_1$-sequences (see Proposition \ref{prop:2sec3}) to show that the heredity of the weak-AFPP for subsets of $C$ it is equivalent to $C$ do not contain $\ell_1$-sequences (see Theorem \ref{thm:FSWAFPP}). The proof of this fact relies essentially on two fundamental results in metrizable LCS. The first one is a slight generalization of the classical Rosenthal's $\ell_1$-theorem to the setting of metrizable LCS (see Theorem \ref{thm:1sec3}). The second one is the Fr\'echet-Urysohn property of the space $(\overline{C-C},w)$, where $w$ denotes the weak-topology of $X$, (see Lemma \ref{lem:FSWAFPP}). This was previously established  by Kalenda \cite{Kalenda} in the Banach spaces setting. As a direct consequence, the weak-AFPP is proved for metrizable LCS without $\ell_1$-sequences (see Corollary \ref{cor:1sec3}). In the final section, we establish the weak-AFPP for certain non-metrizable LCS.
 Firstly, we give an example (see Example \ref{ex:1sec4}) to illustrate that the assumption of metrizability cannot be dropped in the statement of Theorem \ref{thm:1sec3}. We then prove a result in the spirit of Theorem \ref{thm:FSWAFPP} by assuming the existence of metrizable locally convex topologies on Hausdorff LCS compatible with the duality (see Theorem \ref{thm:nonmetrizable}). As a byproduct, we obtain another characterization of the Fr\'echet-Urysohn property in LCS (see Proposition \ref{prop:FU}). With this result in hands, we are able to get an improvement of Theorem 2.4 of \cite{Kalenda}. We then finish the paper with some remarks and some new questions concerning the weak-AFPP in non-metrizable spaces.

\section{On the $\sigma(X,Z)$-approximate fixed point property in HTVS}
\label{sec:2}
Let $X$ be a Hausdorff topological vector space (abbreviated HTVS) with topology $\tau$, and let $C$ be a bounded, closed convex subset of $X$. Let us recall that a nonempty subset $A$ of $X$ is called bounded if for each zero-neighborhood $U$ in $X$, there exists a positive real number $r$ such that $A\subset rU$. Our first main result is the following theorem which is a generalization of \cite[Lemma 2.1]{Barroso-Lin}. Its proof is a simplified version of that given in \cite{Barroso-Lin}. In \cite{Barroso-Lin} the authors used paracompactness of metric spaces and Brouwer's fixed point theorem. The present proof avoids paracompactness. This allows the theorem to be very general -- it works for nonempty bounded convex sets in Hausdorff topological vector spaces with no more assumptions. Furthermore, we mention that the second part of its statement yields a generalization of \cite[Theorem 2.2]{Barroso-Lin} proved in Proposition~\ref{prop:sepdual} below.

\begin{thm}\label{prop:1sec2} Let $(X,\tau)$ be a Hausdorff topological vector space, $Z$ a subspace of its topological dual $X^*$, and let $C$ be a nonempty bounded convex subset of $X$. Assume that $f\colon C\to \overline C$ is a mapping which is $\tau$-to-$\sigma(X,Z)$ sequentially continuous. Then the following hold:
\begin{itemize}
	\item[(i)] $0\in \overline{\{x - f(x)\colon x\in C\}}^{\sigma(X,Z)}$
	\item[(ii)] If, moreover, $Z$ is separable in the strong topology (i.e., the topology of uniform convergence on $\tau$-bounded subsets of $X$), then there is a sequence $(z_n)$ in $C$ such that $z_n-f(z_n)$ converge to $0$ in the topology $\sigma(X,Z)$.
\end{itemize}
\end{thm}

\begin{proof} (i) By the definition of the topology $\sigma(X,Z)$ it is enough to prove the following:

\medskip

{\it Given $x_1^*,\dots,x_n^*\in Z$ and $\epsilon>0$, there is $x\in C$ such that
$|x_i^*(x-f(x))|<\epsilon$ for $i=1,\dots,n$.}

\medskip

So, let us fix $x_1^*,\dots,x_n^*\in Z$ and $\epsilon>0$.
Equip the space $\mathbb{R}^n$ with the max-norm $\|\cdot\|_\infty$
and define the mapping $\Phi\colon \overline C\to \mathbb{R}^n$ by $\Phi(x)=(x_i^*(x))\sb{i=1}\sp{n}$. It is clear that $\Phi$ is a continuous linear mapping. Since $C$ is bounded in $X$, $\overline C$ is bounded as well and hence the set $\Phi(\overline C)$ is  bounded in $\mathbb{R}^n$. It follows that $\Phi(\overline C)$ is totally bounded. Consider the set $U=(-\epsilon/2,\epsilon/2)^n\subset \mathbb{R}^n$. (It is an open ball with respect to the max-norm.)
Next choose a finite set $A\subset \Phi(C)$ so that $\{ U+q\colon q\in A\}$ is an open cover of $\overline{\Phi(C)}$ (note that $\Phi(\overline C)\subset\overline{\Phi(C)}$ as $\Phi$ is continuous). For each $q\in A$ fix some $y_q\in C$ with $\Phi(y_q)=q$. Set $L=\{y\sb{q}\colon q\in A\}$ and $K=\co L$.
Then $K$ is a finite-dimensional compact convex subset of $C$.
 Now for each $x\in K$, fix $z\sb x\in L$  such that $\Phi(f(x))\in \Phi(z_x)+U$. Then
\[
|x_i^*(z\sb x - f(x))|<\epsilon/2, \qquad x\in K, i=1,\dots,n.
\]
Moreover, the restriction $f|_K$ is $\tau$-to-$\sigma(X,Z)$ continuous as $K$ is metrizable. Further, $\Phi$ is $\sigma(X,Z)$-continuous, hence the composed mapping $\Phi\circ f|_K$ is $\tau$-continuous.
Therefore we can, for each $x\in K$, choose a $\tau$-open neighborhood $U\sb x$ of $x$ (relatively in $K$) such that for any $y\in U\sb x$,
\[
|x_i^*(f(y)-f(x))|<\epsilon/2,\qquad i=1,\dots,n.
\]
 It follows that $\{U\sb x\colon x\in K\}$ is an open covering of $K$. Since $K$ is compact, there are $x_1,\dots,x_m\in K$ such that $\{U_{x_i}\colon i=1,\dots,m\}$ is a cover of $K$. Fix a partition of unity $\{\phi_i\colon i=1,\dots,m\}$ on $K$ dominated by $\{U_{x_i}\colon i=1,\dots,m\}$.
Then the mapping $F\colon K\to K$ given by
\[
F(y) =\sum\sb{i=1}^m\phi\sb i(y)z\sb {x_i},\quad y\in K,
\]
is continuous. By Brouwer's fixed point theorem, it has a fixed point $z\in K$. If $\phi_i(z)\neq 0$, then $z\in U\sb {x_i}$ and hence
\[
|x_j^*(f(z)-f(x_i))|<\epsilon/2, \qquad j=1,\dots,n.
\]
This in turn implies that
\begin{align*}
|x_j^*(z-f(z))|&=\Big| \sum\sb{i=1}^m\phi\sb i(z)x_j^*(z\sb {x_i} -f(z))\Big|\\
&\leq  \sum\sb{i=1}^m\phi\sb i (z)\big( |x_j^*(z\sb {x_i} -f(x_i))|+|x_j^*(f(x_i)-f(z))|\big)\\
&< \epsilon,
\end{align*}
for $j=1,\dots,n$. This completes the proof.

(ii) Let $\{x^*\sb i\}$ be a  strongly dense sequence in $Z$.
By  the assertion (i) we can find for any $n\in\mathbb{N}$  a point $z\sb{n}$ in $C$ so that
\[
|x\sb i^*(z\sb{n} - f(z\sb{n}))|<\frac1n,
\]
for all $i=1,\dots, n$. Then for all integer $i\geq 1$, $|x\sb i^*(z\sb n -f(z\sb n))|\to 0$ as $n\to\infty$. The denseness of $\{x\sb i^*\}$ in the strong topology on $Z$ implies $z\sb n - f(z\sb n)\to 0$ with regarding the topology $\sigma(X,Z)$. This completes the proof.
\end{proof}

As an immediate consequence we get an easy proof of a well-known result of K.~Fan (see \cite{Fan}).

\begin{cor}\label{cor:1sec2}
Let $X$ be a Hausdorff topological vector space such that its topological dual $X^*$ separates the points of $X$. (This is satisfied, for example, if $X$ is locally convex.) Let $C\subset X$ be a nonempty compact convex set. Then each continuous mapping $f:C\to C$ has a fixed point.
\end{cor}

\begin{proof}
Set $A=\{x-f(x):x\in C\}$. Then $A$ is compact
as the image of $C$ by a continuous map $x\mapsto x-f(x)$. So, $A$ is also weakly compact. Since $X^*$ separates points of $X$, the weak topology is Hausdorff and hence $A$ is weakly closed.  By the previous theorem $0$ belongs to the weak closure of $A$, hence $0\in A$, i.e., $f$ has a fixed point.
\end{proof}

\begin{rem} We refer the reader to \cite{Jafari-Sehgal} for other comments on Ky Fan's theorem.
\end{rem}
In the following proposition we collect some situations in which the assertion (ii) of Theorem~\ref{prop:1sec2} can be applied.

\begin{pro}\label{prop:sepdual} Let $X$ be a normed space. Then the following statements hold true:
\begin{itemize}
	\item[(i)] Assume that the completion of $X$ is an Asplund space. Let $\tau$ be a locally convex topology on $X$ compatible with the duality. Then $(X,\tau)$ has the weak-AFPP.
	\item[(ii)] If $(X^*,w^*)$ is $\aleph_0$-monolithic (i.e., each separable subset of $(X^*,w^*)$ has countable network), then $(X^*,\|\cdot\|)$ has the $\sigma(X^*,X)$-AFPP.
\end{itemize}
\end{pro}

\begin{proof}
Let us first make the following easy observation: If $X$ is a HTVS, $C\subset X$ a nonempty closed convex bounded set and $f:C\to C$ a continuous mapping, then
there is a nonempty separable closed convex set $D\subset C$ with $f(D)\subset D$.
Indeed, take any $x_0\in C$ and set $D_0=\{x_0\}$. For $n\in\mathbb{N}$ define by induction $D_n=\co(D_{n-1}\cup f(D_{n-1}))$.
Finally, set
\[
D=\overline{\bigcup_{n=0}^\infty D_n}.
\]
Then $D$ has the required property. Now let us proceed to the proof itself:

(i) Let $C\subset X$ be a nonempty bounded closed convex set and $f:C\to C$ be a $\tau$-continuous mapping. Let $D\subset C$ be nonempty $\tau$-separable closed convex set with $f(D)\subset D$. Then $D$ is clearly norm-separable. Indeed, let $S\subset D$ be a countable $\tau$-dense set. Denote by $S'$ the norm-closed convex hull of $S$. Then $S'$ is norm-separable. Moreover, as it is a closed convex set, it is also weakly closed by a consequence of Hahn-Banach separation theorem. Hence it is $\tau$-closed as well, so in particular $D\subset S'$. It follows that $D$ is norm-separable.

Therefore the closed linear span of $D$ is norm-separable as well. So, we can without loss of generality suppose that $X$ is separable. By our assumption $X^*$ is separable,
we can thus conclude by Theorem~\ref{prop:1sec2}-(ii).

(ii) It is enough to show that each nonempty separable closed convex bounded subset of $X^*$ has the $\sigma(X^*,X)$-AFPP. Let $C\subset X^*$ be such a set.  Set $Y= C_\perp$ and denote by $Z$ the quotient space $X/Y$. Denote by $q$ the canonical quotient map $q:X\to Z$. The adjoint map $q^*:Z^*\to X^*$ is an isometric injection which is, moreover, weak*-to-weak* homeomorphism. The image $q^*(Z^*)$ is equal to $Y^\perp=(C_\perp)^\perp$, which is (by the bipolar theorem) the weak* closed linear span of $C$. It follows that
$q^*(Z^*)$ is weak*-separable, hence the weak* topology of $Z^*$ has countable network. Therefore the dual ball $(B_{Z^*},w^*)$ is metrizable, thus $Z$ is separable.
By Theorem~\ref{prop:1sec2}-(ii) we get that $Z^*$ has the $\sigma(Z^*,Z)$-AFPP. As $q^*$ is both an isometry and weak*-to-weak* homeomorphism, we get that $C$ has the $\sigma(X^*,X)$-AFPP.
\end{proof}

We conclude this section with two instructive examples.

\begin{example} Let $X$ be a Banach space. Then $X^*$ has the $\sigma(X^*,X)$-AFPP in the following cases:
\begin{itemize}
	\item $X$ is separable.
	\item $X$ is weakly compactly generated. In particular, $X=c_0(\Gamma)$ for any set $\Gamma$ or $X=L^1(\mu)$ for a $\sigma$-finite measure $\mu$.
	\item $X$ is weakly Lindel\"of determined.
\end{itemize}
\end{example}

\begin{rem}
We recall that a Banach space $X$ is called {\it weakly compactly generated}
if there is a weakly compact subset $K\subset X$ whose linear span is dense in $X$. Basic properties of this class of Banach spaces can be found for example in \cite[Section 1.2]{fabian}.
Further, $X$ is {\it weakly Lindel\"of determined} provided there is $M\subset X$ with dense linear span such that for each $x^*\in X^*$ there are only countably many $x\in M$ with $x^*(x)\ne 0$. We refer also the reader for example to \cite{K-survey} for complements on these notions.
\end{rem}

\begin{rem}
Any separable space is weakly compactly generated and any weakly compactly generated space is weakly Lindel\"of determined (see, e.g., \cite[Theorem 1.2.5]{fabian}). On the other hand, if $X$ is weakly Lindel\"of determined, then $(X^*,w^*)$ is $\aleph_0$-monolithic. Indeed, it is an easy consequence of the definitions that any bounded separable subset of $(X^*,w^*)$ is metrizable. Therefore the previous example can be proved.
\end{rem}

\begin{example}
\begin{itemize}
	\item If $X=c_0$, then $X^*=\ell_1$ has the $\sigma(X^*,X)$-AFPP, but does not have the weak AFPP.
	\item If $X=\ell_\infty$, then $X^*$ does not have the $\sigma(X^*,X)$-AFPP.
\end{itemize}
\end{example}

\begin{proof} As $c_0$ is separable, by the previous example $c_0^*$ has the $\sigma(X,X^*)$-AFPP. Further, $\ell_1$ does not have the weak-AFPP (see, e.g., Theorem~\ref{thm:FSWAFPP} below).

The space $X=\ell_\infty$ is a Grothendieck space, i.e., the weak and weak* convergences of sequences in $X^*$ coincide. So, if $X^*$ had $\sigma(X^*,X)$-AFPP, then it would have also the weak-AFPP. But it is not the case as $X^*$ contains an isometric copy of $\ell_1$.
\end{proof}

\section{The weak-AFPP in metrizable LCS and $\ell_1$-sequences}\label{sec:3}

As we have already mentioned, Theorem~\ref{prop:1sec2} and Proposition~\ref{prop:sepdual} yield a generalization of the main result of \cite{Barroso-Lin}. Another strengthening of this result is given in \cite{Kalenda}, where it is proved, in particular, that a Banach space has the weak-AFPP if and only if it contains no copy of $\ell_1$. The key point in the proof in \cite{Kalenda} is the verification of the Fr\'echet-Urysohn property for certain sets with respect to the weak topology. Recall that a topological space $S$ is called Fr\'echet-Urysohn if the closures of subsets of $S$ are described using sequences, i.e. if whenever $A\subset S$ and $x\in S$ is such that $x\in \overline{A}$, then a sequence $\{x_n\}$ in $A$ can be achieved so that $x_n\to x$. In view of the results of the previous section it is therefore natural to ask whether the results of \cite{Kalenda} can be generalized to the context of locally convex spaces (LCS, in short). The aim of this section is to  show that, indeed, the same results can be proved for metrizable LCS by similar methods.

We briefly recall that the Rosenthal $\ell\sb 1$-theorem and a powerful theorem of Bourgain, Fremlin and Talagrand \cite[Theorem 3F]{BFT} were the two striking tools used in \cite{Kalenda} for getting the Fr\'echet-Urysohn property of the weak closure of the set $C-C$ with regarding the weak topology inherited from $X$. We also observe that there is a generalization of Rosenthal's theorem to Fr\'echet spaces which, it seems, has been firstly obtained by D\'\i az in \cite[Lemma 3]{Diaz}. Thus the starting point for proving promised generalizations is to understand what it means for a sequence in a LCS be equivalent to the unit basis of $\ell\sb 1$.
To clarify it let us first fix some notation. We denote by $\ell_1^0$ the subspace of $\ell_1$ formed by elements with only finitely many nonzero coordinates. Then the following definition is natural.

\begin{definition}\label{def:1sec3}
Let $X$ be a topological vector space and $(x_n)$ a sequence in $X$. We say that $(x_n)$ is an {\it $\ell_1$-sequence} if the mapping $T_0:\ell_1^0\to X$ defined by
\begin{equation}\label{eq-T0}
T_0((a\sb i)_{i=1}^\infty)=\sum_{i=1}^\infty a\sb ix\sb i,\qquad (a_i)\in\ell_1^0
\end{equation}
is an isomorphism of $\ell_1^0$ onto $T_0(\ell_1^0)$.
\end{definition}

It is clear that any $\ell_1$-sequence is automatically bounded.
In normed spaces this definition coincides with the standard one which is witnessed by the following well-known proposition:

\begin{pro}\label{prop:1sec3} Let $(X,\|\cdot\|)$ be a normed space and $(x\sb n)\subset X$ a bounded sequence. The following are equivalent:
\begin{itemize}
\item[(i)] There is constant $M>0$ such that $\|\sum\sb{i=1}^n a\sb ix\sb i\|\geq M \sum\sb{i=1}^n|a\sb i|$, for any $n\in\mathbb{N}$ and any choice of $a_1,\dots, a_n\in \mathbb{R}$.
\item[(ii)] $(x_n)$ is an $\ell_1$-sequence.
\end{itemize}
If $X$ is complete, then these conditions are equivalent to the following:
\begin{itemize}
\item[(iii)] The mapping $T\colon \ell\sb 1\to X$ defined by $T((a\sb i))=\sum\sb i a\sb ix\sb i$ is a well defined isomorphism of $\ell\sb 1$ onto its image in $X$.
\end{itemize}
\end{pro}

\begin{rem}\label{rem:1sec3} Observe that if $X$ is not complete, then in general (iii) needs not follow from (ii). Indeed, $X=\ell_1^0$ contains an $\ell_1$-sequence but does not contain a copy of $\ell\sb 1$.
\end{rem}

For locally convex spaces we have an analogous result. We recall that a locally convex space $X$ is said to be {\it sequentially complete} if each Cauchy sequence in $X$ converges (cf. \cite[page 210]{kothe}).

\begin{pro}\label{prop:2sec3} Let $X$ be a LCS and $(x\sb n)$ a bounded sequence in $X$. The following are equivalent:
\begin{itemize}
	\item[(i)] There is a continuous seminorm $p$ on $X$ such that
	\begin{equation*}
p\left(\sum_{i=1}^n a_i x_i\right)\geq  \sum_{i=1}^n |a\sb i|,\quad n\in\mathbb{N},a_1,\dots,a_n\in\mathbb{R}.
\end{equation*}
\item[(ii)] $(x_n)$ is an $\ell_1$-sequence.
\end{itemize}
If $X$ is sequentially complete, then these conditions are equivalent to the following:
\begin{itemize}
\item[(iii)] The mapping $T\colon \ell\sb 1\to X$ defined by $T((a\sb i))=\sum\sb i a\sb ix\sb i$ is a well defined isomorphism of $\ell\sb 1$ onto its image in $X$.
\end{itemize}
\end{pro}

\begin{proof} Let $T_0:\ell_1^0\to X$ be defined by \eqref{eq-T0}. As $(x_n)$ is bounded and $X$ is locally convex, it is easy to check that $T_0$ is continuous.

Further, if (i) holds, then $T_0$ is clearly one-to-one and $T_0^{-1}$ is continuous. This proves (i)$\Rightarrow$(ii).

Conversely, suppose that (ii) holds. Set
$$U=T_0\left(\left\{x\in \ell_1^0:\|x\|_{\ell_1}<1\right\}\right).$$
As $T_0$ is an isomorphism, $U$ is an absolutely convex open subset of $T_0(\ell_1^0)$. We can find $V$, an absolutely convex neighborhood of $0$ in $X$ such that $V\cap T_0(\ell_1^0)\subset U$. Let $p$ be the Minkowski functional of $V$. Then $p$ is a continuous seminorm witnessing that (i) holds.
This proves (ii)$\Rightarrow$(i).

Now suppose that $X$ is sequentially complete.
As $T\sb 0$ is continuous and linear, it is uniformly continuous and hence it maps Cauchy sequences to Cauchy sequences. In particular, the mapping $T\sb 0$ can be uniquely extended to a continuous linear mapping $T\colon \ell\sb 1\to X$. This is obviously the mapping described in (iii). As $\ell_1^0$ is dense in $\ell_1$, we get (ii)$\Leftrightarrow$(iii).
\end{proof}

Now we are able to formulate the following theorem, which is a generalization of \cite[Theorem 1.2]{Kalenda} to the context of metrizable locally convex spaces.

\begin{thm}\label{thm:FSWAFPP} Let $X$ be a metrizable LCS and $C\subset X$ a nonempty closed convex bounded subset of $X$. The following assertions are equivalent.
\begin{itemize}
	\item[(1)] Each nonempty closed convex subset of $C$ has the weak-AFPP.
	\item[(2)] $C$ contains no $\ell_1$-sequence.
\end{itemize}
\end{thm}

As an immediate consequence we get the following corollary.

\begin{cor}\label{cor:1sec3}
Let $X$ be a metrizable LCS not containing any $\ell\sb 1$-sequence. Then $X$ has the weak-AFPP.
\end{cor}

\begin{proof}[Proof of the implication (1)$\Rightarrow$(2) of Theorem~\ref{thm:FSWAFPP}]
Let us suppose by contradiction that (2) does not hold. Fix an $\ell_1$-sequence $(x_n)$ in $C$, and denote by $D$ the closed convex hull and by $Y$ the closed linear span of the set $\{x_n:n\in\mathbb{N}\}$. Let $T_0:\ell_1^0\to X$ be defined by \eqref{eq-T0}. By our assumption $T_0$ is an isomorphism of $\ell_1^0$ onto $T_0(\ell_1^0)$. Denote by $S_0$ its inverse. Then $S_0$ is an isomorphism of $T_0(\ell_1^0)$ onto $\ell_1^0$. In particular, $S_0$ maps Cauchy sequences to Cauchy sequences. Thus $S_0$ can be uniquely extended to a continuous linear mapping $S:\overline{T_0(\ell_1^0)}\to\ell_1$. Note that $\overline{T_0(\ell_1^0)}=Y$ and that $S$ is an isomorphism of $Y$ onto $S(Y)\subset \ell_1$. As $S$ is linear, it is also a weak-to-weak homeomorphism.

We claim that the set $D$ does not have the weak-AFPP. Suppose on the contrary that it has the weak-AFPP. Then $S(D)$ has the weak-AFPP as well. But then, by Schur's theorem, $S(D)$ has the AFPP. By \cite[Theorem 1.1]{Lin-Sternfeld} we get that $S(D)$ is totally bounded. But it cannot be the case as $S(D)$ contains the canonical basis of $\ell_1$. This completes the proof.
\end{proof}


To prove the converse implication we need some more results. The following theorem is a variant of Rosenthal's $\ell_1$-theorem. Its proof is a slight refinement of the proof of \cite[Lemma 3]{Diaz}.

\begin{thm}\label{thm:1sec3} Let $X$ be a metrizable LCS. Then each bounded sequence in $X$ contains either a weakly Cauchy subsequence or a subsequence which is an $\ell\sb 1$-sequence.
\end{thm}

\begin{proof} Let $(\|\cdot\|\sb n)$ be a sequence of seminorms generating the topology of $X$. Without of loss of generality we may assume that $\|x\|_n\leq \|x\|_{n+1}$ for all $n$ and $x\in X$ (cf. \cite[page 205]{kothe}). Let $U_n=\{ x\colon \|x\|\sb n<1\}$ and let $B_n=U_n^0$ be the polar of $U_n$. Assume that $(x_k)$ is a bounded sequence in $X$ such that no its subsequence is an $\ell_1$-sequence. For $n=0,1,2,\dots$ we construct a sequence $(x_k^n)$ inductively as follows. Set $x_k^0=x_k$ for all $k\in \mathbb{N}$. Assume that for a given $n\in\mathbb{N}$ the sequence $(x_k^{n-1})$ has been defined. By Rosenthal's theorem \cite[Theorem 1]{rosenthal} one of the following possibilities takes place (elements of $X$ are viewed as functions on $B_n$):
\begin{itemize}
	\item[(i)] $(x_k^{n-1})$ has a subsequence which is equivalent to the $\ell_1$-basis on $B_n$.
	\item[(ii)] $(x_k^{n-1})$ has a subsequence which pointwise converges on $B_n$.
\end{itemize}
Let us show that the case (i) cannot occur. Indeed, suppose that (i) holds. Let
$(y_k)$ be the respective subsequence. The equivalence to the $\ell_1$ basis on $B_n$ means that there is some $C>0$ such that
$$\left\|\sum_{i=1}^k a_i y_i\right\|_n\ge C\sum_{i=1}^k|a_i|$$
for each $k\in\mathbb{N}$ and each choice $a_1,\dots,a_k\in\mathbb{R}$. By Proposition~\ref{prop:2sec3} $(y_k)$ is an $\ell_1$-sequence in $X$, which is a contradiction.

Thus the possibility (ii) takes place. Denote by $(x_k^n)$ the respective  subsequence. This completes the inductive construction.

Take the diagonal sequence $(x_k^k)$. It is a subsequence of $(x_k)$ which pointwise converges on $B_n$ for each $n\in\mathbb{N}$. Moreover, if $x^*\in X^*$ is arbitrary, then there is $n$ and $c>0$ such that $cx^*\in B\sb n$. In particular, the linear span of the union of all $B\sb n's$ is the whole dual $X^*$. It follows that the sequence $(x_k^k)$ is weakly Cauchy. The proof is complete.
\end{proof}

We continue by the following lemma which generalizes \cite[Proposition 2.3]{Kalenda}.

\begin{lem}\label{lem:FSWAFPP} Let $X$ be a metrizable LCS and $C\subset X$ a separable bounded set containing no $\ell_1$-sequence. Then the weak closure of the set $C-C=\{x-y: x,y\in C\}$ is Fr\'echet-Urysohn with regarding the weak topology.
\end{lem}

\begin{proof} As the closed linear span of $C$ is separable, we can without loss of generality suppose that $X$ is separable.
Let $(\|\cdot\|\sb n)$, $U_n$ and $B_n$ ($n\in\mathbb{N})$ be as in the proof of Theorem \ref{thm:1sec3}.
Notice that $B\sb n$ is a metrizable weak$^*$ compact subset of $X^*$. Moreover, the linear span of the union of all $B\sb n's$ is the whole dual $X^*$ (see the end of the proof of Theorem~\ref{thm:1sec3}). Let now $P$ be the topological sum of the spaces $(B\sb n, w^*)$, $n\in\mathbb{N}$. Then $P$ is a Polish space. Denote by $G\colon P\to X^*$ the canonical mapping of $P$ onto the union of all $B\sb n's$. Then $G$ is continuous from $P$ to $(X^*,w^*)$. Define a mapping $H\colon X\to \mathbb{R}^P$ by the formula $H(x)(p)=G(p)(x)$. Then $H$ is a homeomorphism of $(X,w)$ onto $H(X)$ equipped with the pointwise convergence topology. Moreover, the functions from $H(X)$ are continuous on $P$.

Let $A=H(C-C)$. We claim that each sequence from $A$ has a pointwise convergent subsequence. To show that it is enough to observe that each sequence in $C-C$ has weakly Cauchy subsequence. Indeed, let $(z_n)$ be a sequence in $C-C$. Then $z_n=x_n-y_n$ for some $x_n,y_n\in C$. As $C$ contains no $\ell_1$-sequence, by Theorem~\ref{thm:1sec3} we get a weakly Cauchy subsequence $(x_{n_k})$ of $(x_k)$. Applying Theorem~\ref{thm:1sec3} once more we get a weakly Cauchy subsequence $(y_{n_{k_l}})$ of $(y_{n_k})$. Then $(z_{n_{k_l}})$ is a weakly Cauchy subsequence of $(z_{n})$.

 Thus $A$ is relatively countably compact in $B_1(P)$, which is the space of all Baire-one functions on $P$ equipped with the topology of pointwise convergence. By the Theorem 3F of \cite{BFT}, the closure of $A$ in $B\sb 1(P)$ is compact and Fr\'echet-Urysohn. In particular, the weak closure of $C-C$ in $X$ is Fr\'echet-Urysohn in the weak topology. The proof is complete.
\end{proof}

The implication $(2)\Rightarrow(1)$ of Theorem \ref{thm:FSWAFPP} follows immediately from the following proposition.

\begin{pro}\label{prop:3sec3} Let $(X,\tau)$ be a metrizable LCS, $C\subset X$ a nonempty convex bounded set which does not contains any $\ell_1$-sequence and $f:C\to\overline{C}$ a $\tau$-to-weak continuous mapping. Then there is a sequence $(x_n)$ in $C$ such that $x_n-f(x_n)$ weakly converge to $0$.
\end{pro}

\begin{proof} First let us find a nonempty separable convex $D\subset C$ with $f(D)\subset\overline{D}$. To do that fix $x_0\in C$ and set $D_0=\{x_0\}$.
Suppose that $D_n\subset C$ is a nonempty separable convex set. Then $f(D_n)$ is a weakly separable subset of $\overline C$. As weakly separable sets are separable, we can find $S_n\subset f(D_n)$ a countable $\tau$-dense set.
As $\tau$ is metrizable and $S_n\subset \overline{C}$, there is a countable set
$T_n\subset C$ with $S_n\subset\overline{T_n}$. Set $D_{n+1}=\co(D_n\cup T_n)$.
Then $D_{n+1}$ is a separable convex subset of $C$ containing $D_n$.
Finally, set $D=\bigcup_{n=0}^\infty D_n$. Then $D$ is a nonempty convex separable subset of $C$ and, moreover,
$$f(D)=\bigcup_{n=0}^\infty f(D_n)\subset\bigcup_{n=0}^\infty \overline{S_n}
\subset\bigcup_{n=0}^\infty \overline{T_n}
\subset\bigcup_{n=0}^\infty \overline{D_{n+1}}\subset\overline{D}.$$
From Theorem~\ref{prop:1sec2} we get $0\in\overline{\{x-f(x)\colon x\in D\}}^{w}$, where $w=\sigma(X,X^*)$. Thus, according to Lemma \ref{lem:FSWAFPP}, there exists a sequence $(x\sb n)$ in $D$ so that $x\sb n - f(x\sb n)\to 0$ in the weak topology. This proves the result.
\end{proof}

\section{Weak-AFPP in non-metrizable LCS}\label{sec:4}

In the previous section we have proved that the results of \cite{Kalenda} can be extended to the framework of metrizable LCS. But if we compare these results to the completely abstract results of Section~\ref{sec:2}, it is natural to ask whether the metrizability assumption is necessary. Let us remark that this assumption was really used several times. In the proof of the implication $(1)\Rightarrow(2)$ of Theorem~\ref{thm:FSWAFPP} we used the Fr\'echet-Urysohn property of the topology of $X$ to extend the mapping $S_0$ to $S$. In the proof of the converse implication we used metrizability twice by assuming that the topology is generated by a sequence of seminorms
-- once in the proof of Theorem~\ref{thm:1sec3} and once in the proof of the implication itself.

In the proof of the implication $(1)\Rightarrow(2)$ of Theorem~\ref{thm:FSWAFPP}
the metrizability assumption can be avoided by noticing that $S_0$ maps Cauchy nets to Cauchy nets. That the metrizability is essential for the converse implication and for Rosenthal's theorem it witnessed by the following example.

Before stating the example we will prove the following lemma which we will use in the proof of the Example~\ref{ex:1sec4} and Theorem~\ref{thm:nonmetrizable} below.
We think that this lemma should be essentially well known, but we do not know any reference. So, we give a proof.  

\begin{lem}\label{positivecone} Let $\Gamma$ be an arbitrary set. Then the norm and weak topologies coincide on the positive cone of $\ell_1(\Gamma)$.
\end{lem}

\begin{proof} Denote by $C$ the positive cone of $\ell_1(\Gamma)$.
Since the weak topology is weaker than the norm one, it is enough to prove that the identity 
of $(C,w)$ onto $(C,\|\cdot\|)$ is continuous.

Let $x\in C$ and $\varepsilon>0$ be arbitrary. Fix a nonempty finite set $F\subset \Gamma$ such that
$$\sum_{\gamma\in F} x(\gamma)>\|x\|-\frac\varepsilon4.$$
Set
$$\begin{aligned}
U&=\left\{y\in C: |y(\gamma)-x(\gamma)|<\frac \varepsilon{4|F|}\mbox{ for }\gamma\in F\right\}, \\
V&=\left\{y\in C: \sum_{\gamma\in\Gamma\setminus F} y(\gamma)-\sum_{\gamma\in\Gamma\setminus F} x(\gamma)<\frac\varepsilon4\right\}.\end{aligned}$$
Then both $U$ and $V$ are weak neighborhoods of $x$ in $C$ (recall that the dual of $\ell_1(\Gamma)$ is represented by $\ell_\infty(\Gamma)$), hence so is $U\cap V$. Moreover, if $y\in U\cap V$, then
$$\begin{aligned}
\|y-x\| & = \sum_{\gamma\in F} |y(\gamma)-x(\gamma)| + \sum_{\gamma\in\Gamma\setminus F} |y(\gamma)-x(\gamma)| 
< \frac \varepsilon4 + \sum_{\gamma\in\Gamma\setminus F} (y(\gamma)+x(\gamma))
\\ & = \frac \varepsilon4 + \sum_{\gamma\in\Gamma\setminus F} (y(\gamma)-x(\gamma)) +
2\sum_{\gamma\in\Gamma\setminus F} x(\gamma) < \frac \varepsilon4 +\frac \varepsilon4 +2\cdot \frac \varepsilon4 =\varepsilon.\end{aligned}$$
This shows that the identity is weak-to-norm continuous at $x$. The proof is completed.
\end{proof}

\begin{example}\label{ex:1sec4}
Let $X=(\ell_1,w)$. Let $(e_n)$ denote the canonical basic sequence.
\begin{itemize}
	\item[(i)] The sequence $(e_n)$ contains neither a weakly Cauchy subsequence nor a subsequence which is an $\ell_1$-sequence.
	\item[(ii)] $X$ contains no $\ell_1$-sequence but does not have the weak-AFPP.
	
	\end{itemize}
\end{example}

\begin{proof}
Let us show first that $X$ contains no  $\ell_1$-sequence.
Suppose that $(x_n)$ is an $\ell_1$-sequence in $X$. Denote by $Y$ its linear span. By the definition of an $\ell_1$-sequence we get that $Y$ is isomorphic to $(\ell_1^0,\|\cdot\|)$, hence it is metrizable. On the other hand, by the definition of $X$ we get that $Y$ is equipped with its weak topology which is not metrizable as $Y$ has infinite dimension.

Further, the sequence $(e_n)$ contains no weakly Cauchy subsequence in $(\ell_1,\|\cdot\|)$. As weakly Cauchy sequences in $(\ell_1,\|\cdot\|)$ and in $(\ell_1,w)$ coincide, we get that $(e_n)$ contains no weakly Cauchy subsequence in $X$. Thus the proof of (i) is completed.

To complete the proof of (ii) it remains to show that $X$ does not have the weak-AFPP.	Let $C$ be the closed convex hull of $\{e_n:n\in\mathbb{N}\}$.
As $C$ is contained in the positive cone of $\ell_1$, the norm and weak topologies coincide on $C$ (by Lemma~\ref{positivecone}). Thus $C$ has the weak-AFPP in $X$ if and only if it has the weak-AFPP in $(\ell_1,\|\cdot\|)$. But it does not have the weak-AFPP in $(\ell_1,\|\cdot\|)$ as it contains an $\ell_1$-sequence when considered in the norm topology.
\end{proof}

Nonetheless, for certain non-metrizable LCS we have the following analogue of Theorem~\ref{thm:FSWAFPP}.

\begin{thm}\label{thm:nonmetrizable}
Let $(X,\tau)$ be a Hausdorff LCS such that there is a metrizable locally convex topology on $X$ compatible with the duality. Let $C$ be a nonempty closed convex bounded subset of $X$. The following assertions are equivalent.
\begin{itemize}
	\item[(1)] Each nonempty closed convex subset of $C$ has the weak-AFPP.
	\item[(2)] Each sequence in $C$ has a weakly Cauchy subsequence.
\end{itemize}
\end{thm}

\begin{proof} Let $\rho$ be a metrizable locally convex topology on $X$ compatible with the duality. As any metrizable locally convex topology is Mackey (see \cite[page 263]{kothe}), we get $\sigma(X,X^*)\subset\tau\subset\rho$.

(2)$\Rightarrow$(1) Let $D\subset C$ be a nonempty closed convex subset of $C$ and $f:D\to D$ be a continuous mapping. It is easy to find a nonempty $\tau$-separable closed convex set $D'\subset D$ with $f(D')\subset D'$ (see the proof of Proposition~\ref{prop:sepdual}). By Theorem~\ref{prop:1sec2} we get
that $0$ belongs to the weak closure of $\{x-f(x):x\in D'\}$.

Further, as $D'$ is $\tau$-separable, it is also $\rho$-separable. By Theorem~\ref{thm:1sec3} $D'$ contains no $\ell_1$-sequence in $(X,\rho)$,
hence by Lemma~\ref{lem:FSWAFPP} the weak closure of $D'-D'$ is Fr\'echet-Urysohn in the weak topology, hence there is a sequence $(x_n)$ in $D'$ such that $x_n-f(x_n)$ weakly converge to $0$.

(1)$\Rightarrow$(2) Suppose that (2) does not hold, i.e. that there is a sequence in $C$ having no weakly Cauchy subsequence. By Theorem~\ref{thm:1sec3} there is a sequence $(x_n)$ in $C$ which is an $\ell_1$-sequence in $(X,\rho)$.
Let $D$ be the closed convex hull of $\{x_n:n\in\mathbb{N}\}$ and $Y$ be the closed linear span of $D$. We note that it does not matter whether we consider closed convex hulls and closed linear spans with respect to the topology $\tau$, $\rho$ or $\sigma(X,X^*)$ as all these topologies have the same dual.

By the proof of the implication (1)$\Rightarrow$(2) of Theorem~\ref{thm:FSWAFPP}
there is a linear mapping $T:Y\to\ell_1$ which is an isomorphism of $(Y,\rho)$ onto $T(Y)$ and, moreover, $(D,\rho)$ does not have the weak-AFPP. We claim that $(D,\tau)$ does not have the weak-AFPP either. This will be done if we show that the topologies $\rho$ and $\tau$ coincide on $D$.

To do that we recall that $T$ is a $\rho$-to-norm isomorphism and weak-to-weak homeomorphism of $Y$ onto $T(Y)$ and, moreover, $T(D)$ is contained in the positive cone of $\ell_1$. It follows from Lemma~\ref{positivecone} that on the positive cone of $\ell_1$ the norm and weak topologies coincide. It follows that $\rho$ and $\sigma(X,X^*)$ coincide on $D$. As $\sigma(X,X^*)\subset\tau\subset\rho$, the proof is complete.
\end{proof}

As a byproduct we obtain the following characterization of the Fr\'echet-Urysohn property in locally convex spaces.

\begin{pro}\label{prop:FU}
Let $(X,\tau)$ be a Hausdorff LCS such that there is a metrizable locally convex topology on $X$ compatible with the duality. The following assertions are equivalent.
\begin{itemize}
	\item[(i)] Any bounded subset of $X$ is Fr\'echet-Urysohn in the weak topology.
	\item[(ii)] Any bounded sequence in $X$ has a weakly Cauchy subsequence.
\end{itemize}
If, moreover, $\tau$ itself is metrizable, then these assertions are equivalent also to the following one:
\begin{itemize}
	\item[(iii)] $X$ contains no $\ell_1$-sequence.
\end{itemize}
\end{pro}

\begin{proof}
Let $\rho$ be a metrizable locally convex topology compatible with the duality.
By Theorem~\ref{thm:1sec3} $(X,\rho)$ contains no $\ell_1$-sequence if and only if $(X,\rho)$ satisfies the condition (ii). Further, the validity of (ii) for $(X,\rho)$ is equivalent to its validity for $(X,\tau)$. It follows that (ii) holds if and only if $(X,\rho)$ contains no $\ell_1$-sequence. In particular,
if $\rho=\tau$, we get (ii)$\Leftrightarrow$(iii).

(ii)$\Rightarrow$(i) Suppose that (ii) holds.
Let $A$ be a bounded subset of $(X,\tau)$ and let $x\in X$ belong to the weak closure of $A$. We need to find a sequence in $A$ weakly converging to $x$.

We first prove it under the additional assumption that $A$ is separable.
Then $A$ is bounded and separable in $(X,\rho)$ as well.
As $(X,\rho)$ contains no $\ell_1$-sequence, by Lemma~\ref{lem:FSWAFPP} we get that the weak closure of $A-A$ is Fr\'echet-Urysohn in the weak topology. Hence, in particular, there is a sequence in $A$ weakly converging to $x$.

To prove the general case it is enough to show that there is a countable set
$C\subset A$ such that $x$ belongs to the weak closure of $C$. In other words, it is enough to show that the weak topology on $X$ has countable tightness. To prove that observe that $(X,\sigma(X,X^*))$ is canonically homeomorphic to a subspace of $C_p(X^*,\sigma(X^*,X))$, which is the space of all continuous functions on the space $(X^*,\sigma(X^*,X))$ equipped with the topology of pointwise convergence.  Further notice that $(X^*,\sigma(X^*,X))$ is $\sigma$-compact -- this follows by the metrizability of $\rho$ as $X^*=\bigcup_{m,n\in\mathbb{N}} mB_n$ using the notation from the proof of Theorem~\ref{thm:1sec3}. Finally, as any finite power of a $\sigma$-compact space is again $\sigma$-compact and hence Lindel\"of, we can conclude by Arkhangel'skii-Pytkeev theorem \cite[Theorem II.1.1]{archan}.

(i)$\Rightarrow$(ii) Suppose that (ii) does not hold. Then there is a sequence $(x_n)$ in $X$ which is an $\ell_1$-sequence in $(X,\rho)$. Let $T_0:\ell_1^0\to X$ be defined as in \eqref{eq-T0}. Let $S$ denote the unit sphere in $\ell_1^0$.
Then $0$ is in the weak closure of $S$ (as $\ell_1^0$ is an infinite dimensional normed space) but it is not the weak limit of any sequence from $S$ (by Schur's theorem). Thus, $0$ is in the weak closure of $T_0(S)$ without being the weak limit of any sequence from $T_0(S)$. Thus $T_0(S)\cup\{0\}$ is a bounded set which is not Fr\'echet-Urysohn in the weak topology.
\end{proof}

As a consequence we get the following improvement of \cite[Theorem 2.4]{Kalenda}.

\begin{cor}\label{cor:1sec4}
Let $X$ be a Banach space. The following assertions are equivalent.
\begin{itemize}
	\item[(1)] $X$ contains no isomorphic copy of $\ell_1$.
	\item[(2)] The closed unit ball $B_X$ is Fr\'echet-Urysohn in the weak topology.
	\item[(3)] There is a sequence $(F_n)_{n=1}^\infty$ of weakly closed sets which are Fr\'echet-Urysohn in the weak topology such that $X=\bigcup_{n=1}^\infty F_n$.
\end{itemize}
\end{cor}

\begin{proof} The equivalence (1)$\Leftrightarrow$(2) follows from Proposition~\ref{prop:FU}. The implication (2)$\Rightarrow$(3) is trivial.
The implication (3)$\Rightarrow$(1) follows from \cite[Theorem 2.4]{Kalenda}
(or, alternatively, (3)$\Rightarrow$(2) follows from the Baire category theorem as in \cite[Theorem 2.4]{Kalenda}).
\end{proof}

Let us remark that the implication (ii)$\Rightarrow$(i) of Proposition~\ref{prop:FU} does not hold for general LCS. Indeed, there are Banach spaces $X$ such that the dual unit ball $B_{X^*}$ is weak* sequentially compact, but it is not Fr\'echet-Urysohn in the weak* topology. In particular,
the dual unit ball is weak* sequentially compact whenever $X$ is Asplund (this follows for example from \cite[Theorem 2.1.2]{fabian}), in particular if $X=C(K)$ with $K$ scattered (see, e.g., \cite[Theorem 1.1.4]{fabian}). On the other hand, $K$ is canonically homeomorphic to a subset of $(B_{C(K)^*},w^*)$, so it is enough to observe that there are scattered compact spaces which are not Fr\'echet-Urysohn. As a concrete example we can take $K=[0,\omega_1]$, the ordinal interval equipped with the order topology ($\omega_1$ is the first uncountable ordinal).

However the following problems seem to be open.

\begin{problems} Let $X$ be a Hausdorff LCS.
\begin{itemize}
	\item Is it true that each bounded sequence in $X$ has a weakly Cauchy subsequence if and only if each bounded {\it separable} subset is Fr\'echet-Urysohn in the weak topology?
	\item Is it true that $X$ has the weak-AFPP if and only if  each bounded sequence in $X$ has a weakly Cauchy subsequence?
\end{itemize}
\end{problems}

As we have seen, both questions have positive answer if $X$ admits a metrizable locally convex topology compatible with the duality. We do not know what happens without this assumption. We conjecture that at least the first question has negative answer.
A candidate for a counterexample could be the space $(X^*,w^*)$ where $X$ is one of the Johnsonn-Lindenstrauss spaces constructed in \cite{JL}.

\end{document}